\documentclass[11pt,leqno]{amsart}
\usepackage{amsmath,amssymb,amsthm}
\usepackage{hyperref}

\usepackage{bbm}

\newcommand{\R}{\mathbb{R}}

\DeclareMathOperator{\diam}{diam\,}
\DeclareMathOperator{\co}{co}

\newcommand{\cconv}{\overline{\co}}

\renewcommand{\geq}{\geqslant}
\renewcommand{\leq}{\leqslant}

\newcommand{\norm}[1]{\left\Vert#1\right\Vert}

\newcommand{\spann}{\operatorname{span}}

\newtheorem{theorem}{Theorem}[section]
\newtheorem{lemma}[theorem]{Lemma}

\newtheorem{proposition}[theorem]{Proposition}
\newtheorem{corollary}[theorem]{Corollary}
\theoremstyle{definition}
\newtheorem{definition}[theorem]{Definition}

\theoremstyle{remark}
\newtheorem{remark}[theorem]{Remark}
\numberwithin{equation}{section}
\newcommand{\abs}[1]{\left\lvert#1\right\rvert}

\def\fnote#1{\footnote}
\def\blacksquare{\hbox{\vrule width 4pt height 4pt depth 0pt}}
\def\ignora#1{}
\def\n3#1{\left\vert  \! \left\vert \! \left\vert \, #1 \, \right\vert \!
  \right\vert \! \right\vert }

\usepackage{accents}
\usepackage[most]{tcolorbox}
\usepackage{comment}

\let\emptyset\varnothing

\newcommand{\Q}{\mathbb{Q}}
\newcommand{\N}{\mathbb{N}}
\newcommand{\G}{\mathbb{G}}
\newcommand{\cP}{\mathcal{P}}
\newcommand{\cPi}{\mathcal{P}_\infty}
\newcommand{\cB}{\mathcal{B}}
\newcommand{\cI}{\mathcal{I}}
\newcommand{\V}{[-1,1]^V}
\newcommand{\ra}{\rightarrow}
\newcommand{\lra}{\leftrightarrow}
\newcommand{\eps}{\varepsilon}
\newcommand{\conv}{\text{co}}
\newcommand{\Fsd}{F_{\sigma\delta}}
\newcommand{\UC}{\widehat{UC}}
\newcommand{\LDP}{\widehat{LD2P}}
\newcommand{\DdosP}{\widehat{D2P}}
\newcommand{\SDP}{\widehat{SD2P}}
\newcommand{\DP}{\widehat{DP}}
\newcommand{\DLDP}{\widehat{DLD2P}}
\newcommand{\DDdosP}{\widehat{DD2P}}
\newcommand{\LOH}{\widehat{LOH}}
\newcommand{\WOH}{\widehat{WOH}}
\newcommand{\OH}{\widehat{OH}}
\newcommand{\iLDP}{\widetilde{LD2P}}
\newcommand{\iDdosP}{\widetilde{D2P}}
\newcommand{\iSDP}{\widetilde{SD2P}}
\newcommand{\iDP}{\widetilde{DP}}
\newcommand{\iDLDP}{\widetilde{DLD2P}}
\newcommand{\iDDdosP}{\widetilde{DD2P}}
\newcommand{\iLOH}{\widetilde{LOH}}
\newcommand{\iWOH}{\widetilde{WOH}}
\newcommand{\iOH}{\widetilde{OH}}
\def\interior#1{\accentset{\circ}{#1}}
\def\Vu#1{\left( \V \right)^{#1}}

\begin{document}

\author{ Gin\'es L\'opez-P\'erez }\address{Universidad de Granada, Facultad de Ciencias. Departamento de An\'{a}lisis Matem\'{a}tico, 18071-Granada
(Spain)} \email{ glopezp@ugr.es}

\author{ Esteban Martínez Vañó }\address{Universidad de Granada, Facultad de Ciencias. Departamento de An\'{a}lisis Matem\'{a}tico, 18071-Granada
(Spain)} \email{ emv@ugr.es}

\author{ Abraham Rueda Zoca }\address{Universidad de Granada, Facultad de Ciencias. Departamento de An\'{a}lisis Matem\'{a}tico, 18071-Granada
(Spain)} \email{ abrahamrueda@ugr.es}
\urladdr{\url{https://arzenglish.wordpress.com}}



\thanks{ This research has been supported  by MCIU/AEI/FEDER/UE Grant PID2021-122126NB-C31 and by Junta de Andaluc\'{\i}a Grant FQM-0185.}

\thanks{The first autor research has been also supported by MICINN (Spain) Grant CEX2020-001105-M (MCIU, AEI)}

\thanks{ The second author research has also been supported by MCIN/AEI/10.13039/501100011033 and FSE+ Grant PRE2022-101438. }

\thanks{The third author research has been also supported by Fundaci\'on S\'eneca: ACyT Regi\'on de Murcia grant 21955/PI/22 and by Generalitat Valenciana project CIGE/2022/97.
}

\title{Computing Borel complexity of some geometrical properties in Banach spaces}

\begin{abstract}
We compute the Borel complexity of some classes of Banach spaces such as different versions of diameter two properties, spaces satisfying the Daugavet equation or spaces with an octahedral norm. In most of the above cases our computation is even optimal, which completes the research done during the last years around this topic for isomorphism classes of Banach spaces.
\end{abstract}

\maketitle

\markboth{L\'OPEZ-P\'EREZ, MART\'INEZ VAÑ\'O,  RUEDA ZOCA}{COMPUTING BOREL COMPLEXITY OF SOME GEOMETRICAL ...}

\section{Introduction}
The relation between Banach space theory and descriptive set theory is not new and has proven to be very effective in the study of universality properties as \cite{dodos} perfectly shows. However, this paper will continue a more recent investigation line focused on computing the Borel complexity of well known classes of Banach spaces. 

This kind of studies began with \cite{bos1} and \cite{bos} where B. Bossard codified the class of all separable Banach spaces with a polish structure, allowing him to use the tools from descriptive set theory to show that many important families of separable Banach spaces, such as reflexive spaces, spaces with separable dual, spaces which do not contain an isomorphic copy of $\ell_1$ or spaces with the Radon-Nikodym property, are complete conanalytic (see also \cite{boslo}). His approach consisted in choosing any universal separable Banach space, like $X=C(2^\omega)$, and consider the Effros-Borel structure in $F(X)$, the family of its closed subsets. In that way he obtained a standard Borel space, that is, a measurable space whose $\sigma$-algebra is the Borel $\sigma$-algebra for some polish topology on it, in such a way that the subset $SB(X)$ of all closed linear subspaces is also a standard Borel space and, as $X$ is universal, it can be used to encode the class of all separable Banach spaces. By means of this codification B. Bossard presented numerous relevant results, but his approach have a limitation: one can only decide if some class is Borel or not, but it is meaningless to specify its Borel complexity as there isn't any canonical polish topology on $SB(X)$.

To supplement this defect in the codification G. Godefroy and J. Saint-Raymond introduced in its recent work \cite{gs-r} a family of topologies over $F(X)$, called \textit{admissible topologies}, which preserve the results of B. Bossard as they have the Effros-Borel structure as its $\sigma$-algebra, but they do so in a canonical way in the sense that one can jump between different admissible topologies by means of $\Sigma_2^0$-maps, so the Borel hierarchy is practically preserved. Even more recently, M. Cúth, M. Dole\v{z}al, M. Doucha and O. Kurka introduced in \cite{cddk1} a new codification which is equivalent to that of the admissible topologies but has some very desirable properties that makes it very useful for computing complexity classes. Using this approach, in \cite{cddk2} they achieved to compute the exact complexity class of many more classical families of Banach spaces such as the isometry classes of the $\ell_p$ and $L_p$ spaces, $c_0$ or the Gurari\u{\i} space, $QSL_p$-spaces o superreflexive spaces.

As it can be seen, most of the work has been done in the study of isomorphic properties or in the isometry class of particular spaces, but, as we will show with this paper, this new codification given in \cite{cddk1} is also very useful to compute the complexity class of some geometrical notions. In fact, based on previous results, it is natural to study the class of spaces with some diameter $2$ property, because they are the extreme case of spaces failing the Radon-Nikodym property, and that of the spaces with an octahedral norm, because G. Godefroy and B. Maurey proved in \cite{gm} that its isomorphism class is precisely that of the spaces which contains an isomorphic copy of $\ell_1$.

In that way, after a brief section for background and notation, we will use B. Bossard's results to prove in Section \ref{section:isomorclasses} that the classes of spaces with a renorming satisfying any of these properties are complete analytic, as it is expected by the relation with the RNP and the containment of  isomorphic copies of $\ell_1$.

In contrast with the above results, in Section \ref{section:compleB} we will show a method that gives us access to the dual of a separable Banach space inside $\cB$ and will allow us to prove that the isometry classes for all diameter $2$ properties are $G_\delta$-complete in the space $\cB$. This method does not seem to be adaptable to the space $\cPi$ so, in Section \ref{section:complepinfinito}, we will use a more geometrical approach to compute the optimal descriptive complexity class of all these classes for $\cPi$.

Finally, in Section \ref{section:compleocta} we will prove that the class of spaces with a (locally) octahedral norm is also $G_\delta$-complete and that the class of weakly octahedral spaces is $F_{\sigma\delta}$ in any of the spaces $\cPi$ or $\cB$.

\section{Background and notation}

Throughout the paper we will only deal with real Banach spaces. We will denote by $B_X$ and $S_X$ the unit ball and the unit sphere of $X$. We will also denote by $X^*$ the topological dual of $X$. By a \textit{slice} of the unit ball of $X$ we refer a set of the following form
$$S(B_X,f,\alpha):=\{x\in B_X:f(x)>1-\alpha\},$$
where $f\in S_{X^*}$ and $\alpha>0$.

\subsection{The seminorms spaces}

We will briefly present here the codification introduced by M. Cúth, M. Dole\v{z}al, M. Doucha and O. Kurka in \cite{cddk1} to give an idea of the type of spaces we will work with and to fix notation, but we refer to the cited paper for a complete exposition.

We will denote by $V$ the space of rational sequences with finite support and by $\cP$ the subset of $\R^V$ consisting of all the seminorms over $V$. Since this set is closed in $\R^V$, it is a polish space with the inherited product topology. It is then possible to associate to each $\mu \in \cP$ a Banach space $X_\mu$ by extending $\mu$ to $\bar{\mu}$, its unique seminorm extension in $c_{00}$, and after that passing to the natural quotient with the set
$$N_\mu=\{x \in c_{00}: \bar{\mu}(x)=0\}$$
to obtain a norm which can be completed to get $(X_\mu, \hat{\mu})$, the desired Banach space. 

Furthermore, to deal with infinite-dimensional spaces, we will consider two more subspaces of seminorms, $\cPi$ and $\cB$, which are the subsets of $\cP$ such that $X_\mu$ is infinite-dimensional and, in the case of $\cB$, such that $(c_{00}, \bar{\mu})$ is a normed space (so for $\cB$ we don't need to use the quotient to obtain $X_\mu$).

As a consequence of the above discussion, given a separable Banach space $(X, \norm{\cdot})$ with a dense sequence $\{x_n\}$, we can obtain a seminorm $\mu \in \cP$ such that $X \equiv X_\mu$ by defining for every $v=\sum a_n e_n \in V$ (with $\{e_n\}$ the canonical basis of $c_{00}$)
$$\mu(v) = \norm{\sum_{n=1}^\infty a_n x_n}.$$
The same is true for $\cPi$ and $\cB$, so all these spaces can be seen as codifications for the class of all separable Banach spaces.

We will also denote by $\bar{V}$ the quotient class of $V$ under the considered seminorm $\mu$, which will be a countable dense subset of $X\mu$. Observe that in $\cB$ we can directly work with $V$ so, in such a way, we have a universal countable dense set for every separable Banach space.

Regarding the notation, we will use the calligraphic $\cI$ when we are working indistinctly with any of the three previous seminorms spaces. Given then a property $P$ over the class of Banach spaces, we will denote its associated \textit{isomorphism class} in $\cI$ by $\widetilde{P}_{\cI}$, that is,
\begin{align}
    \widetilde{P}_{\cI}=\{\mu \in \cI: X_\mu \text{ is isomorphic to a space with the property } P\}.
\end{align}
Analogously, we will denote its \textit{isometry class} in $\cI$ by $\widehat{P}_{\cI}$. In particular, given a fixed separable Banach space $X$, we will denote its isomorphism (resp. isometry) class by $\langle X \rangle_{\simeq}^{\cI}$ (resp. $\langle X \rangle_{\equiv}^{\cI}$). If it is clear in which space we are working the subindex will be omitted.

\subsection{Diameter \texorpdfstring{$2$}{2} and Daugavet properties}

In 1963 I. K. Daugavet proved in \cite{daugavet} that the operator equation
\begin{align}\label{daug}
    \norm{\text{Id} + T} = 1 + \norm{T},
\end{align}
which is now known as Daugavet equation, holds for compact operators in $C[0,1]$. From that point onwards this equation has been extended to larger families of operators and it was proven in \cite{kssw} that just from holding for rank $1$ operators we obtain that it holds for weakly compact operators. For that reason, it is said that a Banach space $X$ has the Daugavet property if equation (\ref{daug}) holds for any rank $1$ operator $T: X \ra X$. Furthermore, in this same paper, the authors gave a geometric characterization of this property which we will take as definition:

\begin{definition}
    A Banach space $X$ has the \textit{Daugavet property} (DP) if for every $\eps>0$, $x \in S_X$ and $S$ a slice of $B_X$ there exists a $y \in S$ such that $\norm{x-y} > 2-\eps$.
\end{definition}

From this geometrical characterization it is easy to see that every slice of the ball of a Banach space with the Daugavet property has diameter $2$, so it fails to have the Radon-Nikodym property in a rotund manner. But much more can be said because a space with the DP satisfies any of the following diameter $2$ properties which can be seen as extreme cases of spaces failing the RNP, the point of continuity property and strong regularity respectively:

\begin{definition}
    A Banach space $X$ has:
    \begin{enumerate}
        \item The \textit{local diameter $2$ property} ($LD2P$) if every slice of $B_X$ has diameter 2.
        \item The \textit{diameter $2$ property} ($D2P$) if every non-empty relatively weakly open subset of $B_X$ has diameter $2$.
        \item The \textit{strong diameter $2$ property} ($SD2P$) if every convex combination of slices of $B_X$ has diameter $2$. 
    \end{enumerate}
\end{definition}

All this properties have been widely studied (see e.g. \cite{ahntt,blradv,rueda}) and have proven to be an interesting subject of study in the geometry of Banach spaces on its own. Furthermore, we can add two more intermediate properties, introduced in \cite{blrdiametral}, between this diameter $2$ properties and the Daugavet property:

\begin{definition}
    A Banach space $X$ has:
    \begin{enumerate}
        \item The \textit{diametral local diameter $2$ property} (DLD2P) if for every $\eps >0$, every slice $S$ of $B_X$ and every $x \in S \cap S_X$ there exists a $y \in S$ such that $\norm{x-y} > 2-\eps$. 
        \item The \textit{diametral diameter $2$ property} (DD2P) if for every $\eps >0$, every non-empty relatively weakly open subset $U$ of $B_X$ and every $x \in U \cap S_X$ there exists a $y \in U$ such that $\norm{x-y} > 2-\eps$. 
    \end{enumerate}
\end{definition}

We refer the reader to \cite{almt21, hlpv, mpr23}.

\subsection{Octahedral norms}

The concept of octahedral norm was introduced by G. Godefroy and B. Maurey in \cite{gm} to characterize those Banach spaces which contains an isomorphic copy of $\ell_1$, but later it was shown to be a dual property for the SD2P, that is, a Banach space has the SD2P if and only if its dual norm is octahedral \cite{BeLoRu}. For that reason, two new octahedral notions were introduced (see \cite{hlp}) in order to serve as dual notions for the LD2P and the D2P whose definitions are the following:

\begin{definition}
    A Banach space $X$ is:
    \begin{enumerate}
        \item \textit{Locally octahedral} (LOH) if for every $\eps >0$ and $x \in X$ there exists a $y \in S_X$ such that for every $t \in \R$
        $$\norm{tx+y} \geq (1-\eps)(\abs{t}\norm{x}+\norm{y}).$$ 
        \item \textit{Weakly octahedral} (WOH) if for every finite-dimensional subspace $E$ of $X$, every $f \in B_{X^*}$ and $\eps >0$, there exists a $y \in S_X$ such that for every $x \in E$
        $$\norm{x+y} \geq (1-\eps)(\abs{f(x)}+\norm{y}).$$
        \item \textit{Octahedral} (OH) if for every finite-dimensional subspace $E$ of $X$ and $\eps >0$, there exists a $y \in S_X$ such that for every $x \in E$
        $$\norm{x+y} \geq (1-\eps)(\norm{x}+\norm{y}).$$
    \end{enumerate}
\end{definition}

We refer the reader to \cite{ll21, lr20, lr21} for more background on octahedral norms.

\section{Isomorphism classes}\label{section:isomorclasses}

B. Bossard proved in \cite[Corollary 3.3]{bos} that the class of spaces with the Radon-Nikodym property and the class of spaces which do not contain an isomorphic copy of $\ell_1$ are complete coanalytic. By the results in \cite{gm}, this implies that the isomorphism class of octahedral spaces is complete analytic and suggest that the same will happen for the diameter $2$ properties and the other octahedral spaces. In this section we will prove that this is the case by following B. Bossard's work applied to the codifications presented in \cite{cddk1}.

To accomplish this we will need to work with the Pe\l{}czy\'{n}ski space, so we recall here its definition:

\begin{definition}
    The Pe\l{}czy\'{n}ski space $\mathbb{P}$ is the unique (up to isomorphism) separable Banach space with basis that contains isomorphic copies of every separable Banach space with basis as complemented subspaces.
\end{definition}

The interested reader can find a proof of the existence and uniqueness of this space in \cite[Theorem 2.d.10]{lt}.

We will need the following technical lemma:

\begin{lemma}\label{Noanal}
    Let $C$ be any subset of $\cI$ such that:
    \begin{enumerate}
        \item $\langle X_\mu \rangle_{\simeq}^{\cI} \subset C$ for every $\mu \in C$,
        \item $C$ contains the isomorphism class of the Pe\l{}czy\'{n}ski space and
        \item $I\setminus C$ contains the class of all reflexive spaces.
    \end{enumerate}
    Then, $C$ is $\Sigma_1^1$-hard in $\cI$.
\end{lemma}

\begin{proof}
    Let's start by setting $\cI = \cP$.

    By \cite[Lemma 2.4]{bos} there exists a Borel map $\varphi: \mathcal{T} \ra SB(C(2^\omega))$ such that:
    \begin{enumerate}
        \item $\varphi(\theta)$ is reflexive if $\theta$ is well founded.
        \item $\varphi(\theta)$ is isomorphic to $\mathbb{P}$, the Pe\l{}czy\'{n}ski space, if $\theta$ is not well founded and,
        \item $\varphi(\theta)$ is finite dimensional if, and only if, $\theta$ is a finite tree.
    \end{enumerate} 
    On the other hand, if $\tau$ is any admissible topology on $SB(C(2^\omega))$, by \cite[Theorem 3.3]{cddk1} there is a continuous map $\Phi: (SB(C(2^\omega)), \tau) \ra \cP$ such that $F \equiv X_{\Phi(F)}$ for every $F \in SB(C(2^\omega))$. As the Borel $\sigma$-algebra of $(SB(C(2^\omega)), \tau)$ is precisely the one given by the Effros-Borel structure, this last two maps allow us to obtain the Borel map $\psi: \mathcal{T} \ra \cP$ given by $\psi = \Phi \circ \varphi$. If we show that
    \begin{align}\label{clave}
        \psi^{-1}(\cP \setminus C)=WF
    \end{align}
    where $WF$ is the set of well founded trees, as $WF$ is complete coanalytic, we will conclude that $C$ is $\Sigma_1^1$-hard as wanted.

    To prove (\ref{clave}) observe that if $\theta$ is a well founded tree, then $\varphi(\theta)$ is reflexive and so is $X_{\psi(\theta)}$ (because we know that $X_{\Phi(\varphi(\theta))} \equiv \varphi(\theta)$), which implies that $\psi(\theta) \in \cP \setminus C$. Reciprocally, if $\theta$ is not well founded, then $\varphi(\theta)$ is isomorphic to $\mathbb{P}$ and so is $X_{\psi(\theta)}$, which implies that $\psi(\theta) \in C$ as we wanted to show.

    If $\cI = \cPi$ we can change $\varphi$ and $\Phi$ in the previous argument by its restrictions to $\mathcal{T}_i$, the subset of infinite trees, and to $SB_\infty(C(2^\omega))$ respectively, and we will obtain the desired result because $\varphi$ maps infinite trees to infinite-dimensional Banach spaces and the set of infinite well founded trees is complete coanalytic in $\mathcal{T}_i$ (because we are just deleting a countable set from the set of well founded trees which is a well known complete coanalytic set). 
    
    Lastly, if $\cI = \cB$ then we can compose the two previous maps for $\cPi$ with the $\Sigma_2^0$-measurable mapping from $\cPi$ to $\cB$ described in \cite[Proposition 3.6]{cddk1} and repeat the same argument.
\end{proof}

As a consequence of Lemma \ref{Noanal} we will obtain the announced result: all the isomorphism classes we are dealing with in this paper are complete analytic in $\cPi$ and $\cB$.

\begin{theorem}
    The isomorphism classes $\iLDP$, $\iDdosP$, $\iSDP$, $\iDLDP$, $\iDDdosP$, $\iLOH$, $\iWOH$, $\iOH$ and $\iDP$ are complete analytic in $\cPi$ and $\cB$.
\end{theorem}

\begin{proof}
    Observe that all the properties we are considering, which are incompatible with the reflexivity, are weaker than the Daugavet property, so it is enough to show that the Pe\l{}czy\'{n}ski space  has an equivalent norm with the Daugavet property and then an application of Lemma \ref{Noanal} concludes that all these isomorphism classes are $\Sigma_1^1$-hard.

    To this end, observe that the space $X=C([0,1], \mathbb{P})$ has the Daugavet property (c.f. e.g. \cite[Example p.81]{werner}) and that $\mathbb{P}$ has an isometric copy in $X$ which is complemented. Then, by the defining property of $\mathbb{P}$ and using that it is complemented in $X$, we obtain that $X$ contains isomorphic copies of every separable Banach space with a basis as complemented subspaces. By the universal property defining $\mathbb{P}$ we conclude that $\mathbb P$ is isomorphic to $X$, which implies that $\mathbb P$ is isomorphic to a Banach space with the Daugavet property, as desired.

    Lastly, observe that any of the isometry classes analogous to the isomorphic ones given here are Borel (see Theorems \ref{D2Pteo}, \ref{CompClasesB}, \ref{OHyLOHteo} and \ref{WOHteo}). On the other hand, the isomorphism relation is analytic (by \cite[Theorem 2.3]{bos} and \cite[Theorem 3.14]{cddk1}). Putting all together we conclude that all this isomorphism classes are analytic.
\end{proof}

These results make one wonder if the same would happen if we consider isometry classes instead of isomorphism ones. Thanks to this new codification we can answer this question and show in the following sections how different, in descriptive complexity terms, the notions of isomorphism and isometry are.

For future use along the text, we will also present here a similar result for isometry classes:

\begin{proposition}\label{NoFsig1}
    The isometry classes $\widehat{LD\delta P}$ (with $0 < \delta \leq 2$), $\DdosP$, $\SDP$, $\DLDP$, $\DDdosP$, $\LOH$, $\WOH$, $\OH$ and $\DP$ are not $F_\sigma$ in $\cI$.
\end{proposition}

In order to prove it we will need an analogue to Lemma \ref{Noanal}, where we will use another universal space, the Gurari\u{\i} space. We include here its definition and refer the interested reader to \cite{gk} for a more complete exposition. 

\begin{definition}
    The Gurari\u{\i} space $\G$ is the unique (up to isometry) separable Banach space such that for every $\eps>0$, $F, E$ finite-dimensional normed spaces with $E \subset F$ and $f:E \ra \G$ an isometric embedding, there exists an extension of $f$ to $F$ such that it is a $(1+\eps)$-isometric embedding.
\end{definition}

\begin{remark}\label{propgurari}
    Observe that by \cite[Theorem 4.1]{cddk1} the isometry class of the Gurari\u{\i} space is a comeagre set in $\cI$. Therefore, if $P$ is any property (closed by isometries) on the class of Banach spaces and $\widehat{P}_\cI$ is a comeagre set in $\cI$, we have that this two isometry classes must intersect which proves that $\G$ has property $P$.
\end{remark}

\begin{lemma}\label{NoFsig}
    Let $C$ be any subset of $\cI$ such that:
    \begin{enumerate}
        \item $\langle X_\mu \rangle_{\equiv}^{\cI} \subset C$ for every $\mu \in C$,
        \item $C$ contains the isometry class of the Gurari\u{\i} space and,
        \item $I\setminus C$ contains the set
    $$\UC_\infty=\{\mu \in \cI: X_\mu \text{ is uniformly convex and infinite-dimensional} \}.$$
    \end{enumerate}
    Then, $C$ is not a $F_\sigma$ set in $\cI$.
\end{lemma}

\begin{proof}
    Firstly observe that, as any separable Banach space is finitely representable in the class of all separable uniformly convex spaces of infinite dimension, we have that
    \begin{align} \label{N1}
        \cI = \{\mu \in \cI: X_\mu \text{ is finitely representable in } \UC_\infty \}.
    \end{align}
    
    Now, if $C$ was a $F_\sigma$ set in $\cI$, then its complement would be a $G_\delta$ set in $\cI$. Furthermore, by \cite[Proposition 2.9]{cddk1} we have that
    $$\overline{\cI \setminus C} \cap \cI = \{\mu \in \cI: X_\mu \text{ is finitely representable in } \cI \setminus C\},$$
    and by (\ref{N1}) and using that $\UC_\infty \subset \cI \setminus C$ we obtain that $\cI \setminus C$ is dense in $\cI$.
    
    So, we have that $\cI \setminus C$ is a $G_\delta$ dense set in $\cI$, but by \cite[Theorem 4.1]{cddk1} the isometry class of the Gurari\u{\i} space is also a $G_\delta$ dense set in $\cI$, so by Baire's category theorem they must intersect. However, by hypothesis the isometry class of the Gurari\u{\i} space is contained in $C$ and we have the desired contradiction.
\end{proof}

\begin{proof}[Proof of Proposition \ref{NoFsig1}]
      Observe that any of this classes contains the isometry class of the Gurari\u{\i} space because it has the Daugavet property (c.f. e.g. \cite[Corollary 4.5]{aln14}) and $\UC_\infty$ is clearly contained in its complement, so by Lemma \ref{NoFsig} we conclude that, in fact, none of these classes can be $F_\sigma$.
\end{proof}

\section{Description of diameter two and Daugavet properties in \texorpdfstring{$\cB$}{B}}\label{section:compleB}

In this section we aim to analyse the complexity class of all the diameter two properties and the Daugavet property in the space $\cB$. The reason for the restriction to $\cB$ is that in this context we may have a good access to the dual of separable Banach spaces. The following, which is of independent interest, is inspired by the construction given in \cite[Section 2.1.2]{dodos}.

Given $\mu \in \cB$ we define the map
  \begin{alignat}{2}\label{Tmu}
  T_\mu: B_{X_\mu^*} &\longrightarrow& [-1,1]^V \\
  f &\longmapsto &\, T_\mu(f): V &\longrightarrow [-1,1] \nonumber\\
  & & u & \longmapsto \left\{ \begin{array}{lcc} 0 & \text{if} & u=0_V \\ \\ \dfrac{1}{\mu(u)} f(u) & \text{if} & u \not = 0_V \\ \end{array} \right. \nonumber
\end{alignat}

We will denote the image of this map by $K_\mu$. Observe then that given $g \in [-1,1]^V$ we have that
\begin{align}\label{caracKmu}
    g \in K_\mu \lra g(0_V)=0 \land \exists f \in B_{X_\mu^*}\, \forall u \in V f(u)= \mu(u)g(u).
\end{align}

It is then easy to see that $T_\mu$ is an homeomorphism from $(B_{X_\mu^*}, w^*)$ onto $K_\mu$ so, in particular, $K_\mu$ is compact.

The key result to obtain the desired complexity classes is the following lemma:

\begin{lemma}\label{Lema4}
    The set
    \begin{align}
        K=\left\{ (\mu, g) \in \cB \times \V: g \in K_\mu \right\}
    \end{align}
    is closed in $\cB \times \V$.
\end{lemma}

\begin{proof}
    Since $\cB \times \V$ is metrizable (because is a polish space) it is enough to show that if $\{(\mu_n, g_n)\}$ is a sequence in $K$ converging to some $(\mu, g) \in \cB \times \V$, then $(\mu, g) \in K$.

    As $g_n \in K_{\mu_n}$ for every $n \in \N$, by (\ref{caracKmu}) we have that $g_n(0_V)=0$ and that there exists some $f_n \in B_{X_{\mu_n}^*}$ such that for every $u \in V$
    \begin{align}\label{Lema4.1}
        f_n(u)=\mu_n(u) g_n(u).
    \end{align}

    We can now define $f:V \ra \R$ such that for every $u \in V$
    \begin{align}
        f(u)=\mu(u) g(u).
    \end{align}
    First of all, $f$ is $\Q$-lineal because given $u, v \in V$ and $p, q \in \Q$, as the topology in $\cB$ and $\V$ coincide with the pointwise topology, we have by (\ref{Lema4.1}) that
        \begin{align*}
            f(pu+qv) & =\mu(qu+pv) g(qu+pv) = \lim_{n \to \infty} \mu_n(qu+pv) g_n(qu+pv) =\\
             & = \lim_{n \to \infty} f_n(pu+qv) = \lim_{n \to \infty} pf_n(u) + qf_n(v) = \\
             & = \lim_{n \to \infty} p \mu_n(u) g_n(u) + q \mu_n(v) g_n(v) = p \mu(u)g(u) + q  \mu(v)g(v) =\\
             & = pf(u)+qf(v).
        \end{align*}
    Furthermore, as $g \in [-1,1]^V$ we obtain that for every $u \in V$
    $$\abs{f(u)} = \mu(u) \abs{g(u)} \leq \mu(u).$$
    The last two properties allow us to extend $f$ to some function $F \in B_{X_\mu^*}$ and observe that for every $u \in V$
    \begin{align}\label{Lema4.2}
        F(u)=f(u)=\mu(u) g(u).
    \end{align}
    Since
    \begin{align}\label{Lema4.3}
        g(0_v)= \lim_{n \to \infty} g_n(0_V)=0,
    \end{align}
    by (\ref{caracKmu}), (\ref{Lema4.2}) and (\ref{Lema4.3}) we conclude that $g \in K_\mu$, so $(\mu, g) \in K$ as we wanted to show.
\end{proof}

With the above result in mind let us prove that all the complexity classes of diameter two and Daugavet properties are $G_\delta$-complete in $\cB$. Let us start with the class $\DdosP$.

\begin{theorem}\label{D2Pteo}
    The  class $\DdosP$ is $G_\delta$-complete in $\cB$.
\end{theorem}

\begin{proof}
    For every $n \in \N$ we define the set
    \begin{align}
        P_n = & \Bigl\{ \left( \mu, \Vec{g} \right) \in \cB \times \Vu{n}: \forall \eps > 0\, \forall \delta > 0\, \forall u \in V \bigl[ \mu(u) > 1 \, \lor \\
         & \exists i \in \{1, \cdots, n\}\, g_i \not \in K_\mu \lor \exists v, w \in V \bigl( \mu(v-w) > 2-\eps \land \mu(v)<1\, \land  \nonumber\\
         & \mu(w) < 1 \land \forall i \in \{1,\cdots, n\} \bigl[ \abs{\mu(u) g_i(u) - \mu(v) g_i(v)}< \delta\, \land \nonumber\\
         & \abs{\mu(u) g_i(u) - \mu(w) g_i(w)}< \delta \bigr] \bigr) \bigr]\Bigr\}. \nonumber
    \end{align}
    Let us begin showing that
    \begin{align}\label{Lema7.40}
        \DdosP = \bigcap_{n=1}^\infty \pi_{\cB}^c (P_n)
    \end{align}
    where $\pi_{\cB}^c$ is the corresponding co-projection over $\cB$.

    If $\mu \in \DdosP$, we must show that given $n \in \N$ and $\Vec{g} \in \Vu{n}$ we have that $\left( \mu, \Vec{g}\right) \in P_n$. Given then $\eps, \delta >0$ and $u \in V$, if $\mu(u) > 1$ or there exists some $i \in\{1, \cdots, n\}$ such that $g_i \not \in K_\mu$, then it is clear that $\left( \mu, \Vec{g}\right) \in P_n$, so we will suppose that this is not the case.

    Consequently,  we will assume that there exist $f_1, \cdots, f_n \in B_{X_\mu^*}$ such that $g_i=T_\mu(f_i)$ for every $i=1, \cdots, n$. Then, we can define the  weak open set of $X_\mu$ 
    \begin{align}
        U=U(u; f_1, \cdots, f_n; \delta):=\left\{x \in X_\mu: \abs{f_i(u-x)}< \delta, \, i=1, \cdots, n\right\}.
    \end{align}
    As we are suppposing that $\mu(u) \leq 1$ it is obvious that $U \cap B_{X_\mu} \not = \emptyset$, and since $X_\mu$ has the $D2P$ there are some $x, y \in U \cap B_{X_\mu}$ such that
    \begin{align}\label{Lema7.43}
        \hat{\mu}(x-y) > 2-\eps.
    \end{align}
    Using that $V \cap \interior{B_{X_\mu}}$ is dense in $B_{X_\mu}$ we can obtain some $v, w \in V$ such that $\mu(v), \mu(w) < 1$ and
    \begin{align}
        \hat{\mu}(x-v) & < \min\left\{ \dfrac{1}{2} \left( \hat{\mu}(x-y) - (2-\eps) \right), \min_{1 \leq i \leq n} \{ \delta - \abs{f_i(x) - f_i(u)}\} \right\},\label{Lema7.44}\\
        \hat{\mu}(y-w) & < \min\left\{ \dfrac{1}{2} \left( \hat{\mu}(x-y) - (2-\eps) \right), \min_{1 \leq i \leq n} \{ \delta - \abs{f_i(y) - f_i(u)}\} \right\},\label{Lema7.45}
    \end{align}
    where $\hat{\mu}(x-y) - (2-\eps) >0$ because of (\ref{Lema7.43}) and $\delta - \abs{f_i(x) - f_i(u)}>0$ for every $i=1, \cdots, n$ because $x \in U$, and the same happens for $y$.

    From this last inequalities we obtain that
    \begin{align}\label{Lema7.46}
        \mu(v-w) \geq \hat{\mu}(x-y) - \hat{\mu}(x-v) - \hat{\mu}(w-y) > 2-\eps
    \end{align}
    and for every $i= 1, \cdots, n$
    \begin{align}\label{Lema7.47}
        & \abs{\mu(u) g_i(u) - \mu(v) g_i(v)} = \abs{f_i(u) - f_i(v)} \leq \abs{f_i(u) -f_i(x)} + \abs{f_i(x) - f_i(v)} \\
        & \leq \abs{f_i(u) -f_i(x)} + \hat{\mu}(x-v) < \delta. \nonumber
    \end{align}
    It is clear that the same holds true if we replace $v$ with $w$ in (\ref{Lema7.47}).
     Puting all together we conclude that  $\left(\mu, \Vec{g}\right) \in P_n$, as desired.

    For the reverse inclusion, let us assume that $\mu \in \pi_{\cB}^c(P_n)$ for every $n \in \N$ and we will show that given a non-empty weakly open set $U$ of $X_\mu$ such that $U \cap B_{X_\mu} \not = \emptyset$ we get 
    \begin{align}
        \diam \left(U \cap B_{X_\mu} \right) = 2.
    \end{align}
    Without lost of generality we can suppose that
    $$U=U(x;f_1, \cdots, f_m; \delta)$$
    for some $x \in B_{X_\mu}$, $\delta>0$ and $f_1, \cdots, f_m \in B_{X_{\mu}^*}$.

    It is enough then to show that given $\eps >0$ there are some $z, y \in U \cap B_{X_\mu}$ such that
    \begin{align}\label{Lema7.51}
        \hat{\mu}(z-y) > 2- \eps.
    \end{align}
    As $V \cap \interior{B_{X_\mu}}$ is dense in $B_{X\mu}$, there exists some $u \in V$ such that $\mu(u)< 1$ and
    \begin{align}\label{Lema7.52}
        \hat{\mu}(x-u) < \dfrac{\delta}{2}
    \end{align}
    Define, for every $i=1, \cdots, m$, the map $g_i=T_\mu(f_i)$ and
    $$\Vec{g}=(g_1, \cdots, g_m) \in \Vu{m},$$
    Since $\mu\in \pi_{\cB}^c(P_m)$, we obtain that $\left(\mu, \Vec{g}\right) \in P_m$. Now, since $\mu(u)<1$ and $g_1, \cdots, g_m \in K_\mu$, there must exist some $v, w \in V$ such that $\mu(v), \mu(w) < 1$, 
    \begin{align}\label{Lema7.53}
        \mu(v-w) > 2-\eps,
    \end{align}
    and, for every $i=1, \cdots, m$, the following holds
    \begin{align}
        & \abs{\mu(u) g_i(u) - \mu(v) g_i(v)}< \dfrac{\delta}{2}, \label{Lema7.54.1}\\
        & \abs{\mu(u) g_i(u) - \mu(v) g_i(w)}< \dfrac{\delta}{2}.\label{Lema7.54.2}
    \end{align}
    It is enough to prove that $v, w \in U$. Indeed, in such case we get by (\ref{Lema7.53}) that (\ref{Lema7.51}) is true as desired. Now, to prove that $u,v\in U$ just observe that, as for every $i=1, \cdots, m$ we have that $g_i=T_\mu(f_i)$,  (\ref{Lema7.52}) and (\ref{Lema7.54.1}) imply
    \begin{align*}
        \abs{f_i(x) - f_i(v)} \leq & \abs{f_i(x)-f_i(u)} + \abs{f_i(u) - f_i(v)} \leq\\
         & \hat{\mu}(x-u) + \abs{\mu(u) g_i(u) - \mu(v) g_i(v)} < \delta.
    \end{align*}
    This proves that $v \in U$ and using (\ref{Lema7.54.2}) in the same way we can also prove that $w \in U$. This proves (\ref{Lema7.40}).

    Now, from (\ref{Lema7.40}) it is enough to show that $P_n$ is a $G_\delta$ in $\cB \times \Vu{n}$ for every $n \in \N$ to obtain that $\DdosP$ is $G_\delta$ in $\cB$. Indeed, if $P_n$ is a $G_\delta$ set, its complement is a $F_\sigma$ set and, as $\Vu{n}$ is compact, we have that $\pi_{\cB}(P_n^c)$ is a $F_\sigma$ set in $\cB$, so
    $$\pi_{\cB}^c(P_n) = \left(\pi_{\cB}(P_n^c)\right)^c$$
    is in fact a $G_\delta$ set in $\cB$ and so is $\DdosP$ by (\ref{Lema7.40}).

    To prove that $P_n$ is $G_\delta$, observe first that we can express
    \begin{align}\label{Lema7.27}
        P_n= & \bigcap_{m, k \in \N} \bigcap_{u \in V} \Biggl[ C_{u,1}^1 \cup \bigcup_{i=1}^n C_i^2 \cup \bigcup_{v, w \in V} \Biggl( C_{v-w, 2-1/m}^1 \cap C_{v,1}^3 \cap C_{w,1}^3\, \cap\\
         & \,\,\, \bigcap_{i=1}^n \left[ C_{i,u,v,1/k}^4 \cap C_{i,u,w,1/k}^4 \right] \Biggr) \Biggr], \nonumber
    \end{align}
    where for any $u, v \in V$, $t \in \R$ and $i \in \{1, \cdots, n\}$
    \begin{itemize}
        \item $C_{u,t}^1 = \left\{\left( \mu, \Vec{g} \right) \in \cB \times \Vu{n}: \mu(u) > t \right\}$,
        \item $C_{i}^2 = \left\{\left( \mu, \Vec{g} \right) \in \cB \times \Vu{n}: g_i \not \in K_\mu \right\}$,
        \item $C_{u,t}^3 = \left\{\left( \mu, \Vec{g} \right) \in \cB \times \Vu{n}: \mu(u) < t \right\}$,
        \item $C_{i,u,v,t}^4 = \left\{\left( \mu, \Vec{g} \right) \in \cB \times \Vu{n}:  \abs{\mu(u) g_i(u) - \mu(v) g_i(v)} < t\right\}$.
    \end{itemize}
    It is easy to see that the first and the two last ones are open sets by showing that its complement is closed using that the topologies are pointwise convergent. To show that the second one is also open, which will prove using (\ref{Lema7.27}) that $P_n$ is $G_\delta$ as wanted, we will show that
    $$D_i=\left( \cB \times \Vu{n} \right) \setminus C_i^2 = \left\{\left( \mu, \Vec{g} \right) \in \cB \times \Vu{n}: g_i \in K_\mu \right\}$$
    is closed.

    Fix $i\in\{1,\ldots, n\}$ and observe that the map $\pi: \cB \times \Vu{n} \ra \cB \times [-1,1]^V$ given by
    $$\pi \left( \mu, \Vec{g}\right) = (\mu, g_i)$$
    is continuous and that
    $$D_i = \pi^{-1}(K)$$
    where $K$ is the set of Lemma \ref{Lema4}. Since $K$ is closed by Lemma \ref{Lema4}d, the continuity of $\pi$ yields that $D_i$ is also closed as wanted.

    Lastly, as $\DdosP$ is not $F_\sigma$ by Corollary \ref{NoFsig1}, we obtain by \cite[22.11]{kech} that it is $G_\delta$-complete as wanted.
\end{proof}

Using a similar argument we can prove the following theorem:

\begin{theorem}\label{CompClasesB}
    The classes $\LDP$, $\SDP$, $\DLDP$, $\DDdosP$ and $\DP$ are all $G_\delta$-complete in $\cB$.
\end{theorem}

\begin{proof}
    The proof is analogous to that of Theorem \ref{D2Pteo}, that is, we express each class as a countable intersection of co-projections of $G_\delta$ sets in $\cB \times \Vu{n}$.

    For $\LDP, \DLDP$ and $\DP$ we have that these classes are equal to $\pi_{\cB}^c(P)$ with $P$ being, respectively, the set: 
    \begin{align}
        P = & \Bigl\{ \left( \mu, g \right) \in \cB \times \V: \forall m \in \N \, \forall \alpha \in \Q \bigl[ g \not \in K_\mu \lor \forall u \in V \bigl( \mu(u) > 1\, \lor\\ 
         & \mu(u)g(u) \leq \alpha \bigr) \lor \exists v, w \in V \bigl( \mu(v-w) > 2-1/m \land \mu(v)<1\, \land  \nonumber\\
         & \mu(w) < 1 \land \mu(v)g(v) > \alpha \land \mu(w)g(w) > \alpha \bigr) \bigr]\Bigr\}, \nonumber
    \end{align}
    \begin{align}
        P = & \Bigl\{ \left( \mu, g \right) \in \cB \times \V: \forall m \in \N \, \forall \alpha \in \Q\, \forall u \in V \bigl[ g \not \in K_\mu \lor \mu(u) > 1\, \lor\\ 
         & \mu(u)g(u) \leq \alpha \lor \exists v \in V \bigl( \mu(u-v) > 2\mu(u)-1/m \land \mu(v)<1\, \land \nonumber\\
         & \mu(v)g(v) > \alpha \bigr) \bigr]\Bigr\}. \nonumber
    \end{align}
    \begin{align}
        P = & \Bigl\{ \left( \mu, g \right) \in \cB \times \V: \forall m \in \N \, \forall \alpha \in \Q\, \forall u \in V \bigl[ g \not \in K_\mu \lor \mu(u) > 1\, \lor\\ 
         & \forall w \in V \bigl( \mu(w) > 1 \lor\mu(w)g(w) \leq \alpha \bigr) \lor \exists v \in V \nonumber\\
         &\bigl( \mu(u-v) > 2\mu(u)-1/m \land \mu(v)<1 \land \mu(v)g(v) > \alpha \bigr) \bigr]\Bigr\}. \nonumber
    \end{align}

    For $\SDP$ and $\DDdosP$ we have that these classes are equal to
    $$\bigcap_{n \in \N} \pi_{\cB}^c(P_n)$$
    with $P_n$ being, respectively, the set:
    \begin{align}
        P_n = & \Biggl\{ \left( \mu, \Vec{g} \right) \in \cB \times \Vu{n}: \forall m \in \N \, \forall \alpha \in \Q^n \Biggl[ \exists i \in \{1, \cdots, n\}\\
        & \bigl( g_i \not \in K_\mu \lor \forall u \in V \bigl[ \mu(u) > 1 \lor \mu(u)g_i(u) \leq \alpha_i \bigr] \bigr) \lor \exists \Vec{v}, \Vec{w} \in V^n \nonumber\\ 
         & \Biggl( \mu\left(\sum_{i=1}^n \dfrac{1}{n} (v_i - w_i)\right) > 2-1/m \land \forall i \in \{1, \cdots, n\} \bigl[\mu(v_i) < 1 \,\land \nonumber\\
         & \mu(w_i) < 1 \land \mu(v_i)g_i(v_i) > \alpha \land \mu(w_i)g_i(w_i) > \alpha \bigr] \Biggr) \Biggr]\Biggr\}, \nonumber
    \end{align}
    \begin{align}
        P_n = & \Bigl\{ \left( \mu, \Vec{g} \right) \in \cB \times \Vu{n}: \forall m \in \N\, \forall k \in \N\, \forall u \in V \bigl[ \mu(u) > 1 \, \lor \\
         & \exists i \in \{1, \cdots, n\}\, g_i \not \in K_\mu \lor \exists v \in V \bigl( \mu(u-v) > 2\mu(u)-1/m \, \land  \nonumber\\
         & \mu(v)<1 \land \forall i \in \{1,\cdots, n\}  \abs{\mu(u) g_i(u) - \mu(v) g_i(v)}< 1/k \bigr) \bigr]\Bigr\}. \nonumber
    \end{align}

    Lastly, by Corollary \ref{NoFsig1} and \cite[22.11]{kech} we obtain that they are all $G_\delta$-complete sets in $\cB$.
\end{proof}

\section{Complexity of diameter \texorpdfstring{$2$}{2} and Daugavet properties}\label{section:complepinfinito}

Observe that by \cite[Proposition 3.6]{cddk1} there exists a $\Sigma_2^0$-measurable map $\Phi: \cPi \ra \cB$ such that $X_\mu \equiv X_{\Phi(\mu)}$ for every $\mu \in \cPi$, so we can immediately establish that any of the diameter $2$ properties or the Daugavet property is $F_{\sigma\delta}$ in $\cPi$. One could then try to mimic the arguments given in the previous section to get their optimal complexity class in $\cPi$ and, although it is in fact possible to define for this space the analogues to the maps $T_\mu$ and the sets $K_\mu$ and $K$, we were unable to prove that this analogue set $K$ is a $F_\sigma$ set (although it is a $G_\delta$ not closed set as we will show in Remark \ref{Knocerrado}). This is essential to ensure that the sets $P_n$ or $P$ used in the proofs are $G_\delta$, which is key to ensuring that the co-projections are Borel and not just coanalytic.

Because of this reason, to study the optimal complexity of this classes in $\cPi$ we will need to use a different approach, mainly based on expressing these properties in a way that only the norm of the space is involved, without any reference to the dual space. We have accomplished to do it for every diameter $2$ property except for the D2P and the DD2P, where the use of the weak topology (and henceforth the dual space), seems to be unavoidable, so we can only give the previous bound for these two properties in $\cPi$.

We will start by studying the complexity of $\LDP$, but in order to also establish the complexity of the class of spaces with a dentable ball, by the obvious relation between the dentability of the ball of a Banach space and the diameter of the slices of its ball, we will work with a slightly more general property: 

\begin{definition}
    Let $\delta$ be a real number with $0 < \delta \leq 2$. We say that a Banach space $X$ has the \textit{local diameter $\delta$ property} ($LD\delta P$) if every slice of $B_X$ has diameter greater than or equal to $\delta$.
\end{definition}

Following the known characterization for the LD2P given in \cite[Lemma 1]{ivakhno}), it is immediate that a Banach space $X$ has the LD$\delta$P iff for every $\eps >0$ we have that
\begin{align}\label{Ivank1}
    B_X \subset \cconv\left\{ \dfrac{x+y}{2}: x,y \in B_X \land \norm{x-y} > \delta - \eps \right\}.
\end{align}

By using this result we can prove the following lemma, which gives a simple enough description of the $LD\delta P$ to compute its complexity class:

\begin{lemma}\label{LD2Plema}
    Let $X$ be a Banach space, $0 < \delta \leq 2$ and $D\subset X$ a dense subset. The following are equivalent:
    \begin{enumerate}
        \item $X$ has the $LD\delta P$.
        \item For every $\eps >0$ and for every $z \in D$
        $$\norm{z} \geq 1 \lor z \in \cconv_{\Q}\left\{ \dfrac{x+y}{2}: x,y \in \interior{B_X} \cap D \land \norm{x-y} > \delta - \eps \right\},$$
        where $\cconv_{\Q}(A)$ is the closure of all rational convex combinations of elements of $A$.
    \end{enumerate}
\end{lemma}

\begin{proof}
    It is immediate to see that (\ref{Ivank1}) is equivalent to
    \begin{align*}
        \forall \eps>0\, \interior{B_X} \subset \cconv\left\{ \dfrac{x+y}{2}: x,y \in B_X \land \norm{x-y} > \delta - \eps \right\},
    \end{align*}
    which is clearly equivalent to
    \begin{align}\label{Ivank2}
        \forall \eps>0\, \forall z \in X \left(\norm{z} \geq 1 \lor z \in \cconv\left\{ \dfrac{x+y}{2}: x,y \in B_X \land \norm{x-y} > \delta - \eps \right\}\right).
    \end{align}
    Given $A \subset X$, by \cite[Lemma 1.4.4]{pirk} is easy to prove that
    \begin{align}\label{cconvrac}
        \cconv(A) = \cconv_{\Q}(A),
    \end{align}
    and then obtaining (\ref{Ivank2}) is equivalent to
    \begin{align}\label{Ivank3}
        \forall \eps>0\, \forall z \in X \hspace*{-0.7mm} \left(\norm{z} \geq 1 \lor z \in \cconv_{\Q}\left\{ \dfrac{x+y}{2}: x,y \in B_X \land \norm{x-y} > \delta - \eps \right\}\right)\hspace*{-0.7mm}.
    \end{align}
    Furthermore, by a simple density argument, (\ref{Ivank3}) is equivalent to
    \begin{align}\label{Ivank4}
        \forall \eps>0\, \forall z \in D \hspace*{-0.7mm} \left(\norm{z} \geq 1 \lor z \in \cconv_{\Q}\left\{ \dfrac{x+y}{2}: x,y \in B_X \land \norm{x-y} > \delta - \eps \right\}\right)\hspace*{-0.7mm}.
    \end{align}
    and using the density of $D \cap \interior{B_X}$ in $B_X$ we can prove that
    \begin{align}\label{Ivank5}
        & \cconv_{\Q}\left\{ \frac{x+y}{2}: x,y \in B_X \land \norm{x-y} > \delta - \eps \right\} = \\
        & \cconv_{\Q}\left\{ \frac{x+y}{2}: x,y \in \interior{B_X} \cap D \land \norm{x-y} > \delta - \eps \right\}.\nonumber
    \end{align}
    By (\ref{Ivank4}) and (\ref{Ivank5}) it is now clear that (\ref{Ivank1}) is equivalent to (2).
\end{proof}

\begin{theorem}\label{LD2Pteo}
    For any $0< \delta \leq 2$ the isometry class $\widehat{LD\delta P}$ is $G_\delta$-complete in $\cI$.
\end{theorem}

\begin{proof}
    By Lemma \ref{LD2Plema} and using that $\bar V$ is dense in $X_\mu$ for any given $\mu \in \cI$, we know that $\mu \in \widehat{LD\delta P}$ iff for every $\eps >0$ and $u \in \bar V$
    \begin{align}\label{LD2P1}
        \hat{\mu}(\bar u) \geq 1 \lor \bar u \in \cconv_{\Q}\left\{\dfrac{\bar x+ \bar y}{2}: \bar x, \bar y \in \interior{B_{X_\mu}} \cap \bar V \land \hat{\mu}(\bar x- \bar y) > \delta - \eps \right\}.
    \end{align}
    Then, by defining for every $n \in \N$ the countable set
    $$Q_n=\left\{\lambda \in \Q^n: \lambda_i > 0, i=1, \cdots, n \land \sum_{i=1}^n \lambda_i = 1\right\}$$
    we obtain that (\ref{LD2P1}) is equivalent to
        \begin{align}\label{LD2P2}
        & \forall m \in \N\, \forall u \in V \Biggl[ \mu(u) \geq 1 \lor \forall p \in \N\, \exists n \in \N\, \exists \lambda \in Q_n\, \exists \Vec{x}, \Vec{y} \in V^n \\
        & \Biggl( \mu\left(u-\sum_{i=1}^n \lambda_i \dfrac{x_i+y_i}{2}\right) < \dfrac{1}{p} \land  \forall i \in \{1, \cdots, n\}\, \mu(x_i) < 1 \, \land  \nonumber\\
        & \mu(y_i) < 1 \land \mu(x_i-y_i) > \delta-\dfrac{1}{m}\Biggr)\Biggr].\nonumber
    \end{align}
    Lastly, defining for $w \in V$ and $t \in \R$ the closed set
    \begin{align}
        C^0_{w,t} & = \{\nu \in \cI: \nu(w) \geq t\}
    \end{align}
    and the open sets
    \begin{align}
        C^1_{w,t} & = \{\nu \in \cI: \nu(w) < t\},\\
        C^2_{w,t} & = \{\nu \in \cI: \nu(w) > t\},
    \end{align}
    we obtain by (\ref{LD2P2})  that
    \begin{align}
        \LDP = & \bigcap_{m \in \N} \bigcap_{u \in V} \Biggl[ C^0_{u,1} \cup \bigcap_{p \in  \N} \bigcup_{n \in \N} \bigcup_{\lambda \in Q_n} \bigcup_{\Vec{x} \in V^n} \bigcup_{\Vec{y} \in V^n} \Biggl( C^1_{u-\sum_{i=1}^n \lambda_i \frac{x_i+y_i}{2}, 1/p}\, \cap \\
        & \bigcap_{i=1}^n \left( C^1_{x_i,1} \cap C^1_{y_i,1} \cap C^2_{x_i-y_i,\delta-1/m} \right) \Biggr)\Biggr]\nonumber
    \end{align}
    showing that, in fact, $\widehat{LD\delta P}$ is $G_\delta$ in $\cI$.

    Finally, to prove that this class is $G_\delta$-complete we just need to recall that $\widehat{LD\delta P}$ is not $F_\sigma$ by Corollary \ref{NoFsig1} and, because of \cite[22.11]{kech}, this implies that it is indeed $G_\delta$-complete as wanted.
\end{proof}

\begin{corollary}
    The class of all spaces which have a dentable ball is $F_{\sigma\delta}$ in $\mathcal{I}$.
\end{corollary}

\begin{proof}
    We have that $X$ fails to have a dentable ball if, and only if, $X$ has the $LD\delta P$ for some $0 < \delta \leq 2$, so the class of spaces that do not have a dentable ball is
    \begin{align}\label{denta}
        \bigcup_{n \in \N} \widehat{LD\delta_n P}
    \end{align}
    with $\delta_n=1/n$.
    
    By Theorem \ref{LD2Pteo} we know that $\widehat{LD\delta_n P}$ are all $G_\delta$ in $\cI$ and because of (\ref{denta}) we conclude that the class of spaces that do not have a dentable ball is $G_{\delta\sigma}$.
\end{proof}

Up to our knowledge, for the SD2P it is not known any purely geometric characterisation so we must introduce a new one which makes no use of the dual space, but before that it is worth noticing that by J. Bourgain's lemma we can equivalently change the convex combinations of slices in the definition by convex combinations of relatively weakly open subsets of the unit ball (c.f. e.g. \cite[Lemma II.1]{ggms}). Furthermore, using \cite[Lemma 1.4.4]{pirk} we can also equivalently change the convex combinations of slices by means of slices or relatively weakly open subsets of the unit ball.

The next lemma is the desired geometric characterization which, even though it does not seem to be very operative, will allows us to express the SD2P only in terms of the norm of the space:

\begin{lemma}\label{SD2Plema}
    Let $X$ be a Banach space. The following are equivalent:
    \begin{enumerate}
        \item $X$ has the $SD2P$.
        \item For every $n \in \N$, $x_1, \cdots, x_n \in B_X$ and $\eps >0$, there exist an $m \in \N$, some $y_{ij} \in B_X$ for $i=1, \cdots, n$, $j=1, \cdots, m$ and $\alpha_1, \cdots, \alpha_m >0$  with $\sum_{j=1}^m \alpha_j =1$ such that for every $i=1, \cdots, n$
        \begin{align}
            \norm{x_i - \sum_{j=1}^m \alpha_j y_{ij}} < \eps
        \end{align}
        and for every $j=1, \cdots, m$
        \begin{align}
            \norm{\sum_{i=1}^n \dfrac{1}{n} y_{ij}} > 1-\eps.
        \end{align}
    \end{enumerate}
\end{lemma}

\begin{proof} 1)$\Rightarrow$2).    Let $\eps>0, n \in \N$ and $x_1, \cdots, x_n \in B_X$ be given and start by assuming that $X$ has the SD2P.

    For any $U \in \mathcal{W}$, where $\mathcal{W}$ is a $w$-open neighbourhood basis of $0$, we have that
    $$\sum_{i=1}^n \dfrac{1}{n} \left[ (x_i + U) \cap B_X \right]$$
    is a mean of relatively weakly open subsets of $B_X$. By \cite[Theorem 3.1]{lmr} we can obtain, for every $i=1, \cdots, n$, a net $\{x_U^i\}_{U \in \mathcal{W}}$ such that $x_U^i \in x_i + U$ for every $U \in \mathcal{W}$, it weakly converges to $x_i$ and it is satisfied that
    \begin{align}\label{caracSD2P1}
        \forall U \in \mathcal{W} \norm{\sum_{i=1}^n \dfrac{1}{n} x_U^i} > 1-\eps.
    \end{align}
    We can then define in $Z=\bigoplus_{i=1}^n X$ (with the maximum norm) the net $\{z_U\}_{U \in \mathcal{W}}$ by $z_U=(x_U^1, \cdots, x_U^n)$ and observe that, because the weak topology on the product is the product of the weak topologies, it weakly converges to $z=(x_1, \cdots, x_n)$. Then, by Mazur's theorem, we have that
    $$z \in \overline{\{z_U: U \in \mathcal{W}\}}^w \subset \overline{\conv\{z_U: U \in \mathcal{W}\}}^w = \cconv\{z_U: U \in \mathcal{W}\},$$
    so there exist an $m \in \N$, some $\alpha_1, \cdots, \alpha_m >0$  with $\sum_{j=1}^m \alpha_j =1$ and $U_1, \cdots, U_m \in \mathcal{W}$ such that
    \begin{align}\label{caracSD2P2}
        \norm{z - \sum_{j=1}^m \alpha_j z_{U_j}}_\infty < \eps.
    \end{align}
    If we denote $y_{ij} = x_{U_j}^i$, by (\ref{caracSD2P1}) and (\ref{caracSD2P2}) the proof of the implication is finished.

    2)$\Rightarrow$1).  Let
    $$C=\sum_{i=1}^n \dfrac{1}{n} S(B_X, f_i, \alpha_i)$$
    a convex combination of slices of $B_X$ and $\eps >0$ be given. Without lost of generality we can assume that all the $\alpha_i$ are equal to some $0 < \alpha < 1$.
    
    We can then take for every $i=1, \cdots, n$ an
    $$x_i \in S\left(B_X, f_i, \left( \dfrac{\alpha}{n} \right)^2 \right)$$
    and $0<r<\eps$ such that for every $i= 1, \cdots, n$
    \begin{align}\label{caracSD2P4}
        f_i(x_i) > \left(1-\left( \dfrac{\alpha}{n} \right)^2  \right) + r.
    \end{align}
    By (2) we have that there exists an $m \in \N$, some $y_{ij} \in B_X$ for $i=1, \cdots, n$, $j=1, \cdots, m$ and $\alpha_1, \cdots, \alpha_m >0$  with $\sum_{j=1}^m \alpha_j =1$ such that for every $i=1, \cdots, n$
        \begin{align}\label{caracSD2P5}
            \norm{x_i - \sum_{j=1}^m \alpha_j y_{ij}} < r
        \end{align}
        and for every $j=1, \cdots, m$
        \begin{align}\label{caracSD2P6}
            \norm{\sum_{i=1}^n \dfrac{1}{n} y_{ij}} > 1-r > 1-\eps.
        \end{align}
        By (\ref{caracSD2P4}) and (\ref{caracSD2P5}) we then obtain that for every $i=1, \cdots, n$
        \begin{align}\label{caracSD2P7}
            \sum_{j=1}^m \alpha_j y_{ij} \in S\left(B_X, f_i, \left( \dfrac{\alpha}{n} \right)^2 \right),
        \end{align}
        and then, that for every $i=1, \cdots, n$ there exists some $j_i \in \{1, \cdots, m\}$ such that
        \begin{align}\label{caracSD2P7.5}
            y_{ij_i} \in S\left(B_X, f_i, \left( \dfrac{\alpha}{n} \right)^2 \right) \subset S\left(B_X, f_i, \dfrac{\alpha}{n} \right) \subset S(B_X, f_i,\alpha).
        \end{align}
        If we can prove that there exists a common $j$ such that (\ref{caracSD2P7.5}) is satisfied for every $i=1, \cdots, n$, then $\sum_{i=1}^n \frac{1}{n} y_{ij} \in C$ and by (\ref{caracSD2P6}) and \cite[Theorem 3.1]{lmr} we would conclude that $X$ has the SD2P as we want to.

        To accomplish this it is enough to show that if we define for $i=1, \cdots, n$
        $$G_i=\left\{j \in \{1, \cdots, m\}: f_i(y_{ij}) > 1-\dfrac{\alpha}{n}\right\},$$
        (which are non-empty by (\ref{caracSD2P7.5})) we have that
        \begin{align}\label{caracSD2P8}
            G=\bigcap_{i=1}^n G_i \not = \emptyset,
        \end{align}
        which will be proved if we show that for every $j=1, \cdots, m$
        \begin{align}\label{caracSD2P9}
            \sum_{j \not \in G_i} \alpha_j < \dfrac{\alpha}{n}.
        \end{align}
        In fact, if (\ref{caracSD2P9}) is true then
        $$\sum_{j \not \in G} \alpha_j \leq \sum_{i=1}^n \sum_{j \not \in G_i} \alpha_j < \sum_{i=1}^n \dfrac{\alpha}{n} = \alpha < 1,$$
        and because $\sum_{j=1}^m \alpha_j = 1$ we conclude that $G \not = \emptyset$ as wanted.

        To prove (\ref{caracSD2P9}) we just need to observe that by (\ref{caracSD2P7}) we have that for $i=1, \cdots, n$
        \begin{align*}
            1-\left( \dfrac{\alpha}{n} \right)^2 & < \sum_{j \in G_i} \alpha_j f(y_{ij}) + \sum_{j \not \in G_i} \alpha_j f(y_{ij}) \leq \sum_{j \in G_i} \alpha_j + \sum_{j \not \in G_i} \alpha_j \left( 1-\dfrac{\alpha}{n} \right)=\\
            & = 1-\dfrac{\alpha}{n} \sum_{j \not \in G_i} \alpha_j,
        \end{align*}
        and reordering this inequality we obtain (\ref{caracSD2P9}) and conclude the proof.
\end{proof}

By density arguments similar to the ones given in the proof of Lemma \ref{LD2Plema} it is easy to conclude from Lemma \ref{SD2Plema} that if $D$ is dense in $X$ then it has the SD2P iff
\begin{align}\label{largaSD2P}
    & \forall \eps >0\, \forall n \in \N \,\forall x \in D^n \Biggl[ \exists i \in \{1, \cdots, n\} \norm{x_i} \geq 1\, \lor\\
    & \exists m \in \N\, \exists y \in \left( D^m \right)^n\, \exists \alpha \in Q_m \Biggl( \forall i \in \{1, \cdots, n\}\,  \forall j \in \{1, \cdots, m\} \nonumber\\ 
    & \norm{y_{ij}} < 1 \land \forall i \in \{1, \cdots, n\} \norm{x_i - \sum_{j=1}^m \alpha_j y_{ij}} < \eps\, \land \nonumber\\
    & \forall j \in \{1, \cdots, m\} \norm{\sum_{i=1}^n \dfrac{1}{n} y_{ij}} > 1 - \eps \Biggr)\Biggr].\nonumber
\end{align}
Using (\ref{largaSD2P}) and similar arguments to the ones given in Theorem \ref{LD2Pteo} we obtain the following theorem.

\begin{theorem}\label{SD2Pteo}
    The isometry class $\SDP$ is $G_\delta$-complete in $\cI$.
\end{theorem}

Banach spaces having the Daugavet property have a similar geometric characterization to the one given in (\ref{Ivank1}) for the LD2P, that is, a Banach space $X$ has the DP iff for every $\eps >0$ and $x \in S_X$ we have that
\begin{align}\label{werner1}
B_X \subset \cconv\{y \in B_X: \norm{x-y} > 2-\eps\}.    
\end{align}

A proof of this characterization can be found in \cite[Corollary 2.3]{werner} and by a simple argument we can eliminate the dependence on $S_X$ obtaining that (\ref{werner1}) is equivalent to
\begin{align}\label{werner2}
\forall \eps>0 \, \forall x \in X \norm{x}B_X \subset \cconv\{y \in \norm{x} B_X: \norm{x-y} > 2\norm{x}-\eps\}.    
\end{align}

This similarity with (\ref{Ivank1}) allows us to prove in an analogous way to Lemma \ref{LD2Plema} the following one:

\begin{lemma}\label{DPlema}
    Let $X$ be a Banach space and $D\subset X$ a dense subset. The following are equivalent:
    \begin{enumerate}
        \item $X$ has the DP.
        \item For every $\eps >0$ and $x, z \in D$
        $$\norm{z} \geq \norm{x} \lor z \in \cconv_{\Q}\left\{ y \in \norm{x} \interior{B_X} \cap D: \norm{x-y} > 2\norm{x}-\eps \right\}.$$
    \end{enumerate}
\end{lemma}

Then, using this lemma we can also obtain in an analogous way of the proof of Theorem \ref{LD2Pteo} that:

\begin{theorem}\label{Daugteo}
    The isometry class $\DP$ is $G_\delta$-complete in $\cI$.
\end{theorem}

\begin{remark}
    Observe that in the proof of the previous theorem we don't use the Gurari\u{\i} space in any way to prove that $\DP$ is $G_\delta$ in $\cI$. Furthermore, by using \cite[Proposition 2.9]{cddk1} and the fact that $\langle C[0,1] \rangle_{\equiv}^{\cI}$ is contained in $\DP$ we easily obtain that it is dense in $\cI$. Then, by Remark \ref{propgurari}, we obtain an alternative proof to \cite[Corollary 4.5]{aln14} that the Gurari\u{\i} space has the Daugavet property.
\end{remark}

Lastly, it is also possible to give a geometric characterization of the DLD2P (see e.g. \cite[Proposition 2.1.7]{pirk}):

\begin{lemma}\label{DLD2Plema}
    Let $X$ be a Banach space. The following are equivalent:
    \begin{enumerate}
        \item $X$ has the $DLD2P$.
        \item For every $\eps >0$ and $x \in S_X$
        $$x \in \cconv\{y \in B_X: \norm{x-y} > 2-\eps\}.$$
    \end{enumerate}
\end{lemma}

One more time, by the same techniques used in the proof of Lemma \ref{LD2Plema}, from the previous lemma we obtain that if $D$ is a dense set in $X$, then $X$ has de DLD2P iff
\begin{align}
    \forall \eps>0\,\forall x \in D\, x \in \cconv_{\Q} \left\{y \in \norm{x}\interior{B_X} \cap D: \norm{x-y} > 2\norm{x} - \eps\right\},
\end{align}
and using the same arguments given in the proof of Theorem \ref{LD2Pteo} we obtain the desired result:

\begin{theorem}\label{DLD2Pteo}
    The isometry class $\DLDP$ is $G_\delta$-complete in $\cI$.
\end{theorem}

\section{Complexity of octahedral spaces}\label{section:compleocta}

In this last section we will study the complexity of the class of the (locally, weakly) octahedral spaces. For LOH and OH spaces we can give a direct geometrical description of this spaces and use the same approach that in Section \ref{section:complepinfinito} to obtain their optimal complexity class. However, as happened with the D2P and the DD2P, for WOH spaces the presence of the dual space seems to be unavoidable.

It turns out (see \cite[Proposition 2.4]{hlp}) that a Banach space is WOH iff for any $\eps>0$, $n \in \N$, $x_1, \cdots, x_n \in S_X$ and $f \in B_{X^*}$, there is a $y \in S_X$ such that for any $t>0$ and $i \in \{1, \cdots, n\}$ we have that
\begin{align}\label{WOHchar}
    \norm{x_i+ty} \geq (1-\eps)\left( \abs{f(x_i)}+t \right),
\end{align}
and it is an open question if the condition $t>0$ can be replaced with $t=1$. If that was the case, it would be possible to use the approach of Section \ref{section:compleB} and prove that the class $\WOH$ is $G_\delta$-complete in $\cB$, but without that assumption this approach does not work because (\ref{WOHchar}) gives a $\Fsd$ complexity, not a $G_\delta$ one which is necessary for that approach to work. For that reason, we will need a different approach for the WOH spaces, mainly seeing them as the dual class of the D2P property, but we will postpone this until the end of the section.

As said, we can give a purely geometrical characterization of LOH and OH spaces (see \cite[Propositions 2.1 (iii), 3.1 (iii)]{hlp}) and using yet again the same arguments from the previous sections (see Lemma \ref{LD2Plema}) we obtain the following lemma:

\begin{lemma}
    Let $X$ be a Banach space and $D \subset X$ a dense subset of $X$. The following assertions are satisfied:
    \begin{enumerate}
        \item $X$ is LOH iff
        \begin{align}
            & \forall \eps>0\, \forall x \in D\, \exists y \in D \bigl[\norm{y} \not = 0 \land \norm{\norm{y}x - \norm{x} y} > \norm{y} (2\norm{x}-\eps)\, \land\\
            & \norm{\norm{y}x + \norm{x}y} > \norm{y} (2\norm{x}-\eps)\bigr]\nonumber
        \end{align}
        \item $X$ is OH iff
        \begin{align}
            & \forall \eps>0\, \forall n \in \N\, \forall x \in D^n\, \exists y \in D \bigl[\norm{y} \not = 0 \land \forall i \in \{1, \cdots, n\} \\
            &  \norm{\norm{y}x_i + \norm{x_i}y} > \norm{y} (2\norm{x_i}-\eps)\bigr]\nonumber
        \end{align}
    \end{enumerate}
\end{lemma}

From this lemma, using an analogous approach to the proof of Theorem \ref{LD2Pteo}, we obtain the desired theorem:

\begin{theorem}\label{OHyLOHteo}
    The isometry classes $\LOH$ and $\OH$ are $G_\delta$-complete in $\cI$.
\end{theorem}

\begin{remark}
    The only difference in this proof from that of Theorem \ref{LD2Pteo} is that here we will get sets of the form
    $$\left\{\mu \in \cI: \bar{\mu}(\mu(v) u \pm \mu(u) v) > \mu(v) \left(2\mu(u) - \dfrac{1}{m}\right)\right\}$$
    where it appears the extension $\bar{\mu}$ of $\mu$ to $c_{00}$, but it is easy to proof using the density of $\Q$ in $\R$ and the pointwise convergence of $\cI$ that their complements are closed, so they are open as wanted. 
\end{remark}

Finally, as we said at the beggining of the section, to get a bound for the complexity of the class of weakly octahedral spaces we will need the following dual characterization given in \cite[Theorem 2.6]{hlp}:

\begin{theorem}\label{CardualWOH}
    Let $X$ be a Banach space. The following are equivalent:
    \begin{enumerate}
        \item $X$ is weakly octahedral.
        \item $X^*$ has the $w^*$-D2P, that is, every non-empty relatively weak* open subset of $B_{X^*}$ has diameter $2$.
    \end{enumerate}
\end{theorem}

Recall that given a Banach space $X$, $F$ a $w^*$-compact subset of $X^*$ and $\eps>0$, the \textit{Szlenk derivative} of $F$ is given by
\begin{align}
    F_{\eps}' = \left\{f \in F: \forall U \left[U \text{ is } w^*\text{-open} \land f \in U \Rightarrow \diam(U \cap F) \geq \eps\right]\right\}.
\end{align}
Observe that $F_{\eps}'$ is $w^*$-closed and is contained in $F$ which is $w^*$-compact, so it is $w^*$-compact. 

The reason to introduce this concept here is that it is immediate that the dual of a Banach space $X$ has the $w^*$-D2P iff
\begin{align}\label{equivszlenk}
    \left( B_{X^*} \right)_2' = B_{X^*},
\end{align}
so using Theorem \ref{CardualWOH} we can study the class of WOH spaces by studying the Szlenk derivative of its dual ball.

In order to clarify this relation in a way which will allow us to study its descriptive complexity we will need the following easy lemma:

\begin{lemma}\label{lemaSzlenk}
    Let $X$ be a Banach space, $Y$ a dual space and $Z \subset X^*$, $W \subset Y$. If there exists map $T:Z \ra W$ which is an isometry and a \\$w^*$-$w^*$-homeomorphism, then for every $\eps >0$ and $A \subset Z$ $w^*$-compact we have that
    \begin{align}
        T(A_{\eps}') = \left( T(A) \right)_{\eps}'.
    \end{align}
\end{lemma}

Now we are ready to prove the main theorem of this section:

\begin{theorem}\label{WOHteo}
    The isometry class $\WOH$ is $\Fsd$ in $\cPi$ and $\cB$.
\end{theorem}

\begin{proof}
    By Theorem \ref{CardualWOH} we know that $\mu \in \WOH$ iff $X_\mu^*$ has the $w^*$-D2P, and by (\ref{equivszlenk}) this is equivalent to
    \begin{align}\label{WOHteo1}
        \left( B_{X_\mu^*} \right)_2' = B_{X_\mu^*}.
    \end{align}

    By \cite[Lemmas 4.7, 4.9]{cddk2} there is a continuous map $\Omega: \cP \ra \mathcal{K}(B_{\ell_\infty}, w^*)$ such that for every $\nu \in \cP$ there exists a map $T_\nu: B_{X_\nu^*} \ra \Omega(\nu)$ which is an isometry and a $w^*$-$w^*$-homeomorphism. Then,  its restriction to $\cPi$ or $\cB$ will also be continuous and, by Lemma \ref{lemaSzlenk}, we obtain that (\ref{WOHteo1}) is equivalent to
    \begin{align}\label{WOHteo2}
        \left( \Omega(\mu) \right)_2' = \Omega(\mu).
    \end{align}
    Lastly, if we denote by $s_2:\mathcal{K}(B_{\ell_\infty}, w^*) \ra \mathcal{K}(B_{\ell_\infty}, w^*)$ the map given by
    $$s_2(K) = K_2'$$
    for every $K \in \mathcal{K}(B_{\ell_\infty}, w^*)$, by (\ref{WOHteo2}) we obtain that
    \begin{align}\label{WOHteo3}
        \WOH = \{\mu \in \cI: s_2(\Omega(\mu))=\Omega(\mu)\}.
    \end{align}

    By \cite[Lemma 4.10]{cddk2} $s_2$ is a $\Sigma_3^0$-measurable map so, as $\Omega$ is continuous, we have that $s_2 \circ \Omega|_{\cI}$ is also $\Sigma_3^0$-measurable, so the map $\varphi: \cI \ra \mathcal{K}(B_{\ell_\infty}, w^*) \times \mathcal{K}(B_{\ell_\infty}, w^*)$ given by
    $$\varphi(\mu) = \left( (s_2 \circ \Omega)(\mu), \Omega(\mu) \right)$$
    for every $\mu \in \cI$, will also be $\Sigma_3^0$-measurable.
    
    Lastly, we just need to observe that by (\ref{WOHteo3}) we have that
    $$\WOH = \varphi^{-1}(\Delta)$$
    with $\Delta = \{(K,K): K \in \mathcal{K}(B_{\ell_\infty}, w^*)\}$ and, as $\Delta$ is closed in the product of $\mathcal{K}(B_{\ell_\infty}, w^*)$, we conclude that in fact $\WOH$ is $\Fsd$ in $\cI$.
\end{proof}

\section{Comments and open questions}

We stated in Section \ref{section:complepinfinito} that we were unable to prove if the analogue of the set $K$ is an $F_\sigma$ set in $\cPi \times \V$, but in the following remark we prove that it is a $G_\delta$ not closed set which, in some way, justifies the search for another approach to compute the complexity classes in the space $\cPi$ as we did in the mentioned section.

\begin{remark}\label{Knocerrado}
The analogue of the functions $T_\mu$ for $\mu \in \cPi$ are defined in the same way but they are null for every $u \in N_\mu$ and $f$ must be evaluated in the corresponding equivalent class. We will denote in this remark by $K_\mu$ the image of this maps.

Then, to prove that the set
$$K=\left\{ (\mu, g) \in \cPi \times \V: g \in K_\mu \right\}$$
is $G_\delta$ it is enough to notice that for every $\mu \in \cPi$ and $g \in \V$ we have that
\begin{align}
    g \in K_\mu \lra &  \forall u,v,w \in V \, \forall p, q \in \Q \\
     & [(pv+qw=u \ra p\mu(v)g(v) + q \mu(w)g(w)= \mu(u)g(u))\, \land \nonumber\\
     & ( \mu(u)=0 \ra g(u)=0)],\nonumber
\end{align}
and we will see that it is not closed by giving an explicit example of a convergent sequence in $K$ whose limit is not in the set.

For every $n \in \N$ we can define a seminorm in $\R^V$ by
\begin{align}
    \mu_n\left( \sum_{m=1}^\infty q_m e_m \right) = \abs{\dfrac{1}{n} q_1 + q_2} + \sum_{m=3}^\infty \abs{q_m},
\end{align}
and observe that $\mu_n \in \cPi \setminus \cB$ because we have that
$$\mu_n\left(e_1- \dfrac{1}{n} e_2\right) =0$$
and the set $\{\overline{e_1}, \overline{e_3}, \cdots, \overline{e_m}\}$ is linearly independent in $X_{\mu_n}$ for every $m \geq 3$.

Then, we can define for every $n \in \N$ the map $f_n: \spann\{\overline{e_m}\}_{m \in \N} \ra \R$ as the unique linear map such that
\begin{align}
    f_n(\overline{e_1}) & =\dfrac{1}{n}\\
    f_n(\overline{e_m}) & = 1, \, \forall m \geq 3,
\end{align}
which can easily be seen to have norm $1$, and extend it by Hahn-Banach to a functional $F_n$ of norm $1$ in $X_{\mu_n}$. Associated with each $F_n$ we define the maps $g_n=T_{\mu_n}(F_n)$, and in that way we obtain the sequence $\{(\mu_n, g_n)\}$ which is contained in $K$ and it converges to the pair $(\mu, g) \in \cPi \times \V$ given by
\begin{align}
    \mu\left( \sum_{m=1}^\infty q_m e_m \right) = \sum_{m=2}^\infty \abs{q_m},
\end{align}
and
\begin{align}
    g\left( \sum_{m=1}^\infty q_m e_m \right) = \left\{ \begin{array}{lcl}
    0 & \text{if} & \forall m \in \N\, q_m=0\\\\
    \dfrac{\sum_{m=2}^\infty q_m}{\sum_{m=2}^\infty \abs{q_m}} & \text{if} & \exists m >1\, q_m \not = 0\\\\
    \dfrac{q_1}{\abs{q_1}} & \text{if} & q_1 \not = 0 \land \forall m > 1 \, q_m = 0
    \end{array} \right.
\end{align}

It is then clear that $(\mu, g) \not \in K$ because in that case $g$ would be null for every $u \in N_\mu$, but we have that $\mu(e_1)=0$ and $g(e_1)=1$. This proves that $K$ is not closed. $\blacksquare$
\end{remark}

For these reasons, we were not able to directly study the isometry classes for the D2P and the DD2P in $\cPi$ and we left as an open question if this two classes are also $G_\delta$ in $\cPi$. In the case of the DD2P this is a very interesting question because a negative answer would prove that the DLD2P and the DD2P are different properties which, up to our knowledge, is an open question posed in \cite[Question 4.1]{blrdiametral}.

Lastly, as mentioned in Section \ref{section:compleocta}, it is open if a weakly octahedral space $X$ can be characterized as one satisfying that for any $\eps>0$, $n \in \N$, $x_1, \cdots, x_n \in S_X$ and $f \in B_{X^*}$, there is a $y \in S_X$ such that for any $i \in \{1, \cdots, n\}$ we have that
\begin{align}
    \norm{x_i+y} \geq (1-\eps)\left( \abs{f(x_i)}+1 \right).
\end{align} 
Because of this, it is also an open question to get the exact complexity for weakly octahedral spaces. We were able to obtain a bound in Theorem \ref{WOHteo} by using Szlenk derivatives, but this method does not seem to be able to produce a better result because of the comment previous to \cite[Lemma 4.10]{cddk2} which asserts that the map $s_2$ used in the proof is almost optimal.

\section*{Acknowledgements}  
We thank Miguel Mart\'in for fruitful conversations on the topic of the paper.

\end{document}